\documentclass[11pt,a4paper,reqno]{amsart}
\usepackage{amsmath, amsthm, amscd, amsfonts, amssymb}

\newtheorem{theorem}{Theorem}[section]
\newtheorem{lemma}[theorem]{Lemma}
\newtheorem{proposition}[theorem]{Proposition}
\newtheorem{corollary}[theorem]{Corollary}
\newtheorem{problem}[theorem]{Problem}
\theoremstyle{definition}
\newtheorem{definition}[theorem]{Definition}

\newtheorem{remark}{Remark}[section]
\numberwithin{equation}{section}

\begin{document}

\noindent {\footnotesize\tiny}\\[1.00in]

\title[]{Monotone and Pseudo-Monotone Equilibrium Problems in Hadamard Spaces}
\maketitle
\begin{center}
{\sf Hadi Khat\texttt{}ibzadeh\footnote{E-mail: $^{1}$hkhatibzadeh@znu.ac.ir, $^{2}$mohebbi@znu.ac.ir.} and Vahid Mohebbi$^2$}\\
{\footnotesize{\it $^{1,2}$ Department of Mathematics, University
of Zanjan, P. O. Box 45195-313, Zanjan, Iran.}}
%E-mail: $^{1}$hkhatibzadeh@znu.ac.ir, $^{2}$sranjbar@znu.ac.ir}
\end{center}
\begin{abstract}
As a continuation of previous work of the first author with S. Ranjbar \cite{kr2} on a special form of variational inequalities in Hadamard spaces, in this paper we study equilibrium problems in Hadamard spaces, which extend variational inequalities and many other problems in nonlinear analysis. In this paper, first we study the existence of solutions of
equilibrium problems associated with pseudo-monotone bifunctions
with suitable conditions on the bifunctions in Hadamard spaces. Then to approximate of an equilibrium point, we
consider the proximal point algorithm for pseudo-monotone
bifunctions. We prove existence of the sequence generated by
the algorithm in several cases in Hadamard spaces. Next, we introduce the resolvent of a bifunction in Hadamard spaces. We prove convergence of the resolvent to an equilibrium point. We also prove
$\bigtriangleup$-convergence of the sequence generated by the proximal point algorithm to an equilibrium
point of the pseudo-monotone bifunction and also the strong
convergence with additional assumptions on the bifunction.
Finally, we study a regularization of Halpern type and prove the
strong convergence of the generated sequence to an equilibrium point
without any additional assumption on the pseudo-monotone bifunction. Some examples in fixed point theory and convex minimization are also presented.
\end{abstract}

\section{\bf Introduction}
\noindent

Let $(X,d)$ be a metric space. A geodesic from $x$ to $y$ is a map
$\gamma$ from the closed interval $[0,d(x,y)] \subset \mathbb{R}$
to $X$ such that $\gamma(0) = x,\ \gamma(d(x,y)) = y$ and
$d(\gamma(t), \gamma(t')) = |t - t'| $ for all $t, t'\in [0,d(x,y)]$.
The image of $\gamma$ is called a geodesic (or metric) segment
joining from $x$ to $y$. When it is unique, this geodesic segment
is denoted by $[x, y]$. The space $(X, d)$ is said to be a
geodesic space if every two points of X are joined by a geodesic,
and X is said to be uniquely geodesic if there is exactly one
geodesic joining $x$ and $y$ for each $x, y\in X$.
A subset $Y$ of $X$ is said to be convex, if for any two points $x,y \in Y$, the geodesic
 joining $x$ and $y$ is contained in $Y$, that is, if $\gamma :[0,d(x,y)]\longrightarrow X$ is a
  geodesic such that $x=\gamma (0)$ and $y=\gamma (d(x,y))$, then $\gamma(t) \in Y, \  \  \forall t \in [0,d(x,y)]$.\\
Let $X$ be a uniquely geodesic space and $Y\subset X$. The convex hull of
$Y$ (denoted by ${\rm conv}(Y)$) is the intersection of all convex subsets of $X$ that contain $Y$. We recall the following lemma from \cite{pap}, which will be used in the next section.

\begin{lemma} \label{conv lemma}
	Let $X$ be unique geodesic metric space and let $A$ be a subset of $X$. We set $C_0(A)=A$ and for every integer $n\geq0$, we let $C_{n+1}(A)$ be the union of all the geodesic segments in $X$ that join pairs of points in $C_n(A)$. Then, the geodesic convex hull ${\rm conv}(A)$ of $A$ is given by
$${\rm conv}(A)=\bigcup_{n\geq0}C_n(A).$$
\end{lemma}
A geodesic triangle $\Delta (x_1, x_2, x_3)$ in a
geodesic metric space $(X, d)$ consists of three points $x_1,
x_2$ and $x_3$ in $X$ (the vertices of $\Delta$) and a geodesic
segment between each pair of vertices (the edges of $\Delta$). A
comparison triangle for the geodesic triangle $\Delta(x_1, x_2,
x_3)$ in $(X, d)$ is a triangle $\bar{\Delta}(x_1, x_2, x_3) $:=
$\Delta(\bar{x_1}, \bar{x_2}, \bar{x_3})$ in the Euclidean plane
$\mathbb{E}^2$ such that $d_{\mathbb{E}^2}(\bar{x_i}, \bar{x_j})
= d(x_i, x_j )$ for $i, j\in \{1, 2, 3\}$, where
$d_{\mathbb{E}^2}$ is the usual metric in $\mathbb{R}^2$. A geodesic space is said to be a CAT(0) space
if all geodesic triangles satisfy the following comparison axiom.
Let $\Delta$ be a geodesic triangle in $X$, and let
$\bar{\Delta}$ be a comparison triangle for $\Delta$. Then,
$\Delta$ is said to satisfy the CAT(0) inequality if for all $x,
y\in \Delta$ and all comparison points $\bar{x}, \bar{y}\in
\bar{\Delta}$,
$$d(x, y) \leq d_{\mathbb{E}^2}(\bar{x},\bar{y}).$$

A complete CAT(0) space is called a Hadamard space. We
now collect some elementary facts about CAT(0) spaces, which will be used in the proofs of
our main results.
\begin{lemma}\label{mid}
Let X be a CAT(0) space and $x, y\in X$ then, for each $t\in [0, 1]$, there exists a unique point $z\in [x, y]$ such that
$d(x, z) = td(x, y)$ and $d(y, z) = (1 - t)d(x, y)$.
\end{lemma}

\begin{proof}
See Lemma 2.1 (iv) of \cite{sdh}.
\end{proof}

We will use the notation $(1 - t)x \oplus ty$ for the unique point $z$ satisfying in the above statement.

\begin{lemma}\label{metric} Let X be a CAT(0) space. Then for all $x, y, z\in X$ and $t, s\in [0, 1]$, we have \\
(i) $d((1 - t)x \oplus ty, z) \leq (1 - t)d(x, z) + td(y, z)$,\\
(ii) $d((1 - t)x \oplus ty, (1 - s)x \oplus sy) = |t - s|d(x, y)$,\\
(iii) $d((1 - t)z \oplus tx, (1 - t)z \oplus ty) \leq td(x, y)$,\\
(iv) $d^2((1 - t)x \oplus ty, z)\leq (1 - t)d^2(x, z) + td^2(y, z)- t(1 - t)d^2(x, y)$.
\end{lemma}

\begin{proof}
(i) see Lemma 2.4 of \cite{sdh},
(ii) see \cite{pch},
(iii) see Lemma 3 of \cite{wak},
(iv) see Lemma 2.5 of \cite{sdh}.
\end{proof}

Berg and Nikolaev in \cite{nik2, nik} introduced the concept of quasi-linearization
along these lines. Let us formally denote a pair $(a, b) \in X \times  X$ by $\overrightarrow{ab}$ and call it a vector. Then quasi-linearization is defined as a map $\langle \cdot, \cdot\rangle  : (X \times X)\times(X \times X)\rightarrow \mathbb{R}$ defined by\\
$$\langle \overrightarrow{ab},\overrightarrow{cd}\rangle
 =\frac{1}{2}\{d^2(a, d) + d^2(b, c) - d^2(a, c) - d^2(b, d)\}  \   \   \    \ (a, b, c, d \in X).$$
It is easily seen that
$\langle \overrightarrow{ab},\overrightarrow{cd}\rangle
 =
\langle \overrightarrow{cd},\overrightarrow{ab}\rangle $,
$\langle \overrightarrow{ab},\overrightarrow{cd}\rangle
 = -\langle \overrightarrow{ba},\overrightarrow{cd}\rangle $
 and $\langle \overrightarrow{ax},\overrightarrow{cd}\rangle
+
\langle \overrightarrow{xb},\overrightarrow{cd}\rangle
 =
\langle \overrightarrow{ab},\overrightarrow{cd}\rangle $
for all $a, b, c, d, x \in X$. We say that $X$ satisfies the Cauchy-Schwarz inequality if
$\langle \overrightarrow{ab},\overrightarrow{cd}\rangle  \leq d(a, b)d(c, d)$
for all $a, b, c, d \in X$. It is known (Corollary 3 of \cite{nik}) that a geodesically connected metric space
is a CAT(0) space if and only if it satisfies the Cauchy-Schwarz inequality.

A Hadamard space $X$ is called flat Hadamard space iff inequality in Part (iv) of Lemma \ref{metric} is equality. A well-known result asserts that a flat Hadamard space is isometric to a closed convex subset of a Hilbert space. It is easy to check that in a flat Hadamard space $X$, for each $x,y,z,u\in X$ and $0\leq\lambda\leq1$,
\begin{equation}\langle \overrightarrow{xy},\overrightarrow{x(\lambda z\oplus (1-\lambda)u)}\rangle =\lambda\langle \overrightarrow{xy},\overrightarrow{xz}\rangle +(1-\lambda)\langle \overrightarrow{xy},\overrightarrow{xu}\rangle .\label{affin inner product}\end{equation}

Let $(X,d)$ be a Hadamard space, $\{x_n\}$ be a bounded sequence
in $X$ and $x\in X$. Let $r(x,\{x_n\})=\limsup d(x,x_n)$. The
asymptotic radius of $\{x_n\}$ is given by $r(\{x_n\})= \inf
\{r(x,\{x_n\})| x\in X \}$ and the asymptotic center of
$\{x_n\}$ is the set $A(\{x_n\})=\{x \in X|
r(x,\{x_n\})=r(\{x_n\})\}$. It is known that in a Hadamard space,
$A(\{x_n\})$ is singleton.

\begin{definition}
A sequence $\{x_n\}$ in a Hadamard space $(X,d)$
$\triangle$-converges to $x\in X$
if $A(\{x_{n_k}\})=\{x\}$, for each subsequence $\{x_{n_k}\}$ of $\{x_n\}$.\\
We denote $\triangle$-convergence in $X$ by
$\overset{\triangle}{\longrightarrow}$
 and the (strong) metric convergence by $\rightarrow$.
\end{definition}

\begin{lemma}\label{compact}
Let $X$ be a complete CAT(0) space. Then, every bounded closed convex subset of $X$ is $\triangle$-compact; i.e. every bounded sequence in it, has a $\triangle$-convergent subsequence.
\end{lemma}

\begin{proof}
Proposition 3.6 of \cite{w-b}.
\end{proof}

Let $X$ be a complete CAT(0) space. Then, every closed convex subset $K$ of $X$ is $\triangle$-closed in the sense that
it contains all $\triangle$-limit points of every $\triangle$-convergent sequence.

A function $f:X\rightarrow]-\infty,+\infty]$ is called:\\ (1) convex
iff
$$f(\lambda x\oplus(1-\lambda)y)\leq\lambda f(x)+(1-\lambda)f(y),\ \ \forall x,y\in X \ \text{and}\  \lambda\in(0,1)$$
(2) quasiconvex iff
$$f(\lambda x\oplus(1-\lambda)y)\leq\max\{f(x),f(y)\},\ \ \forall x,y\in X\ \text{and}\  \lambda\in(0,1)$$
(3) quasiconcave iff $-f$ is quasiconvex.

A function $f:X\rightarrow]-\infty,+\infty]$ is called proper iff $D(f):=\{x\in X:\ f(x)<+\infty\}\neq\varnothing$. $f$ is called lower semicontinuous (shortly, lsc) at $x\in D(f)$ iff $$\liminf_{y\rightarrow x}f(y)\geq f(x)$$ and it is called $\triangle$-lower semicontinuous (shortly $\triangle$-lsc) at $x\in D(f)$ iff $$\liminf_{y\overset{\triangle}{\rightarrow} x}f(y)\geq f(x).$$
It is well known result that each convex and lsc function is $\triangle$-lsc.

Let $K\subset X$ be nonempty. Consider $f:K\times K\rightarrow\mathbb{R}$, $f$ is called a bifunction. An equilibrium problem for $f$ and $K$ as
briefly $EP(f;K)$ consists of finding $x^*\in K$ such that
$$f(x^*,y)\geq0,\ \ \ \forall y\in K.\ \ \ \  (EP)$$
$x^*$ is called an equilibrium point. We
denote the set of all equilibrium points for $(EP)$ by $S(f;K)$.
Each equilibrium problem has a dual, which is named "convex
feasibility problem" (for short, $CFP$). It consists of finding
$x^*\in K$ such that $f(x,x^*)\leq0$, for all $x \in K$. A convex
feasibility problem for $f$ and $K$ is denoted by $CFP(f,K)$.
Equilibrium problems extend and unify several problems in
optimization, variational inequalities, fixed point theory and
many other problems in nonlinear analysis. Here, $K\subset
X$ denotes a nonempty, closed and convex set unless explicitly
states otherwise. Take $o \in X $, where $o$ is an arbitrary but
fixed point ($o$ is called base-point). The following conditions may be used throughout the paper, therefore we denominate them as:\\
$P1$: $f(x,x)=0$ for all $x\in K$.\\
$P2$: $f(\cdot,y):K\rightarrow\mathbb{R}$ is upper semi-continuous  for all $y\in K$.\\
$P3$: $f(x,\cdot):K\rightarrow\mathbb{R}$ is convex and lower semi-continuous  for all $x\in K$.\\
$f$ is called monotone, iff\\
$P4$: $f(x,y)+f(y,x)\leq0$, for all $x,y\in K$.\\
$f$ is called pseudo-monotone, iff\\
$P4^*$: Whenever $f(x,y)\geq0$ with $x,y\in K$ it holds that $f(y,x)\leq0$.\\
$f$ is called $\theta$-undermonotone, iff\\
$P4^{\bullet}$: There exists $\theta\geq0$ such that
$f(x,y)+f(y,x)\leq \theta d^2(x,y)$, for all $x,y\in K$.\\
$f$ is called coercive, iff\\
$P5$: Let $o \in X $ be the base-point. Then  for any sequence
$\{x_k\}\subset K$ satisfying $\lim d(x_k,o)=+\infty$,
 there exists $u \in K$ and $n_0\in \mathbb{N}$ such that $f(x_n,u)\leq 0$, for all $n \geq n_0$.\\

Equilibrium problems for monotone and pseudo-monotone bifunctions have been extensively studied  in Hilbert, Banach as well as in topological vector spaces by many authors (see \cite{bia-sch, cha-chb-ria, com-hir, iu-k-s, al-wi} and many other references). Recently some authors have studied on equilibrium problems in Hadamard manifolds (see \cite{vitto, nn}). In order to extend and unify of the related results from Hilbert spaces and Hadamard manifolds as well as extension of some recent works on variational inequalities and minimization problems in Hadamard spaces (see \cite{kr2, qro, tzh} ), we study monotone and pseudomonotone equilibrium problems in Hadamard space setting.

The paper has been organized as follows. In the sequel of introduction, we present some well-known lemmas in the Hadamard space framework. In Section 2, we study the existence of solutions of equilibrium problems. In Section 3, in order to approximate an equilibrium point, we use an auxiliary problem. Existence of solutions of the auxiliary problem is not guaranteed for bifunctions with usual assumptions $P1,P2,P3,P4,P4^{*},P4^{\bullet}$ in general Hadamard spaces. In this section we study the existence of solutions of the auxiliary problem in several special cases. Section 4 is devoted to introduce the resolvent operator for pseudo-monotone bifunction and its strong convergence to an equilibrium point. In Section 5, we prove $\Delta$-convergence of the proximal point algorithm for pseudo-monotone bifunctions in Hadamard spaces. Since the strong convergence  (convergence in metric) does not occur even in Hilbert space, in Section 6, we prove strong convergence a regularized version of the sequence of Halpern type in Hadamard spaces. Finally in Section 7, some examples and applications will be presented. Now, we present some lemmas that we need them in the next section.

\begin{lemma}\label{lem01}
With conditions $P1$, $P2$ and $P3$, every solution of  $CFP(f,K)$
solves $EP(f,K)$.
\end{lemma}

\begin{proof}
See Lemma 2.4 of \cite{is}
\end{proof}

\begin{corollary}\label{lem12}
If $f$ satisfies $P1$, $P2$, $P3$ and $P4^*$, then $EP(f,K)$ and $CFP(f,K)$ have the same solution set.
\end{corollary}
The following lemma, which is KKM lemma in complete CAT(0) spaces, has been proved on finite dimensional Hadamard manifolds in \cite{vitto}. The proof is similar for complete CAT(0) spaces, but for completeness of the paper, we rewrite the proof in complete CAT(0) spaces.

\begin{lemma}\label{lem kkm}
 Suppose $X$ is a complete CAT(0) space and $K\subset X$. Let $G:K\rightarrow 2^K$ be a mapping such that for each $x\in K$, $G(x)$ is $\triangle$-closed. Suppose that\\
i) $\forall x_1,\cdots,x_m\in K$, ${\rm conv}(\{x_1,\cdots,x_m\})\subset\bigcup_{i=1}^m G(x_i),$\\
ii) there exists $x_0\in K$ such that $G(x_0)$ is $\triangle$-compact,\\
then $\bigcap_{x\in K} G(x)\neq\varnothing$.
\end{lemma}

\begin{proof}
Take $x_1,\cdots,x_m\in K$ and define
$D(\{ x_1,\cdots,x_m\}):=\bigcup_{i=1}^m D_i$,
where $D_1=\{ x_1\}$ and, for any $2\leq j\leq n$,
$D_j=\{z\in\gamma_{x_j,y}| y\in D_{j-1}\}$ such that $\gamma_{x_j,y}$ is the geodesic joining $x_j$ to some $y\in D_{j-1}$.
Therefore $D(\{ x_1,\cdots,x_m\})$ is a closed subset of ${\rm conv}(\{x_1,\cdots,x_m\})$. Let $y_1=x_1$, in the sequel any element $y_k\in D_k \subseteq D(\{ x_1,\cdots,x_m\})$ can be written as
\begin{equation}\label{kkm1}
y_k=\gamma(t_k),
\end{equation}
where $t_k\in[0,1]$ and $\gamma$ is the geodesic joining $x_k$ to some $y_{k-1}\in D_{k-1}$. To each $x_i$, we associate a corresponding vertex $e_i$ of the simplex $\sigma=\langle e_1,\cdots,e_m\rangle \subset \Bbb R^{m+1}$. Let $T:\sigma\rightarrow D(\{ x_1,\cdots,x_m\})$ be the mapping which is defined by induction as the following form:
For $\lambda_1=e_1$, define $T(\lambda_1)=x_1$ and in the sequel, suppose that $1< k\leq m$, if $\lambda_k\in\langle e_1,\cdots,e_k\rangle \backslash\langle e_1,\cdots,e_{k-1}\rangle $, then $\lambda_k=t_k e_k+(1-t_k)\lambda_{k-1}$ for some $t_k\in(0,1]$ and $\lambda_{k-1}\in\langle e_1,\cdots,e_{k-1}\rangle $.
Hence we define
$T(\lambda_k)=\gamma_k(t_k)$,
where $\gamma_k$ is the geodesic joining $x_k$ to $T(\lambda_{k-1})$ and $t_k$ is the unique element in $[0,1]$ such that $T(\lambda_k)=\gamma_k(t_k)$.\\
The equality (\ref{kkm1}) shows that $T(\sigma)$ coincides with $D(\{ x_1,\cdots,x_m\})$. Now we show that $T$ is continuous. For any $j=1,2,$ let $\lambda^j=\sum_{i=1}^m t_i^j e_i\in\sigma$,
for some sequences $\{t_i^j\}_{i=1}^m\subset[0,1]$ satisfying $\sum_{i=1}^m t_i^j=1.$ By definition, we have that $T(\lambda^j)=\gamma_m^j(t_m^j)$, where $\gamma_m^j$ joins $x_m$ to $T(\sum_{i=1}^{m-1} t_i^j e_i)$. Now let $L:={\rm diam}(D(\langle x_1,\cdots,x_m\rangle))$, in turn by  parts (ii) and (iii) of Lemma \ref{metric}, we have:\\
$d(T(\lambda^1),T(\lambda^2))\leq d(\gamma_m^1(t_m^1),\gamma_m^1(t_m^2))+d(\gamma_m^1(t_m^2),\gamma_m^2(t_m^2))$\\
$\leq |t_m^1 -t_m^2|~ d(x_m,T(\sum_{i=1}^{m-1} t_i^1 e_i))+d(T(\sum_{i=1}^{m-1} t_i^1 e_i),T(\sum_{i=1}^{m-1} t_i^2 e_i))$\\
$\leq L|t_m^1 -t_m^2|+ d(T(\sum_{i=1}^{m-1} t_i^1 e_i),T(\sum_{i=1}^{m-1} t_i^2e_i)).$\\
By recursion, we obtain that
$d(T(\lambda^1),T(\lambda^2))\leq L\sum_{i=1}^m |t_i^1-t_i^2|.$\\
This shows that the continuity of $T$.
Consider the closed sets $\{ E_i\}_{i=1}^m$, defined by $E_i:=T^{-1}(D(\{x_1,\cdots,x_m\})\cap G(x_i))$.
Let us prove that for every $I\subset\{1,\cdots, m\}$,
\begin{equation}\label{kkm2}
{\rm conv}(\{e_i| i\in I\})\subset \bigcup_{i\in I} E_i.
\end{equation}
Indeed, let $\lambda=\sum_{j=1}^k t_{i_j}e_{i_j}\in {\rm conv}(\{e_{i_1},\cdots,e_{i_k}\})$, with $\{t_{i_j}\}\subset[0,1]$
such that $\sum_{j=1}^k t_{i_j}=1$. Since, by the hypothesis:
$$T(\lambda)\in D(\{ x_{i_1},\cdots,x_{i_k}\})\subseteq {\rm conv}(\{x_{i_1},\cdots,x_{i_k}\})\subseteq\bigcup_{n=1}^k G(x_{i_n}),$$
then there exists $j\in\{1,\cdots,k\}$ for which $T(\lambda)\in G(x_{i_j})\bigcap D(\{ x_{i_1},\cdots,x_{i_k}\})$ and,
consequently, $\lambda\in E_{i_j}$. By applying $KKM$ lemma to the family $\{E_i\}_{i=1}^m$,
we get existence of a point $\hat{\lambda}\in {\rm conv}(\{e_1,\cdots,e_m\})$ such that
$\hat{\lambda}\in\bigcap_{i=1}^m E_i$, so $T(\hat{\lambda})\in\bigcap_{i=1}^m G(x_i)$.
We have already proved that the family of $\Delta$-closed sets $\{ G(x)\cap G(x_0)\}_{x\in K}$
 has the finite intersection property. Since $G(x_0)$ is $\triangle$-compact, it implies that
$\bigcap_{x\in K} G(x)=\bigcap_{x\in K} (G(x_0)\cap G(x))\neq \varnothing.$
\end{proof}

\section{\bf Existence of Solutions}

In this section, we are going to study existence of the solutions of equilibrium problems in complete CAT(0) spaces. In \cite{iu-k-s} Iusem, Kassay and Sosa proved existence of the solutions of pseudo-monotone equilibrium problems in Hilbert spaces. Now, we want to extend their results to Hadamard spaces. We assume that $X$ is a Hadamard space and $K
\subset X$ is closed and convex. Let $o \in X$ be the basepoint
and for each $n \in \mathbb N$, set $K_n=\{x\in K |
d(o,x)\leq n\}$. Since $K_n$ is nonempty for sufficiently large
$n$, without loss of generality, we may assume that $K_n$ is
nonempty for all $n\in \mathbb N$. Suppose that $f$ satisfies
$P1$, $P2$ and $P3$. We define for each $y\in K$,
$$L_f(n,y):=\{x\in K_n | f(y,x)\leq 0\}.$$
By applying Lemma \ref{lem01} with $K_n$ instead of $K$, we conclude that $\underset{y\in
K_n}{\bigcap }L_f(n,y)\subseteq \{x\in K_n | f(x,y)\geq0,\forall
y \in K_n\}$, i.e. each solution of the convex feasibility
problem restricted to $K_n$  is a solution of the equilibrium
problem restricted to $K_n$. Take $K_n^\circ \subset K$, the
intersection of $K$ with the open ball of radius $n$ around $o$,
i.e. $K_n^\circ =\{x\in K | d(o,x)< n\}$. We need the following
technical lemmas for the existence result.

\begin{lemma}\label{lem21}
Let $f$ satisfy $P1$, $P2$ and $P3$. If for some $n\in \mathbb
N$ and some $\bar{x}\in\underset{y\in K_n}{\bigcap }L_f(n,y)$
there exists $\bar{y} \in K_n^\circ$ such that
$f(\bar{x},\bar{y})\leq 0$, then  $f(\bar{x},y)\geq0$ for all
$y\in K$.
\end{lemma}

\begin{proof}
A trivial extension of Lemma 3.7 of \cite{is} to geodesic spaces.
\end{proof}

\begin{definition}
$f:K\times K\rightarrow\mathbb{R}$ is called properly quasi-monotone, if for every finite set $A$ of $K$ and
every $y \in {\rm conv}(A)$, $\underset{x\in A}{\min} f(x,y)\leq 0.$
\end{definition}

\begin{lemma}\label{lem22}
If $f$ satisfies one of the following conditions\\
i) $f(\cdot,y)$ is quasi-concave for all $y\in K$ and $P1$ holds,\\
ii) $f(x,\cdot)$ is quasi-convex for all $x\in K$ and $P1$ and $P4^*$ hold,\\
then $f$ is properly quasi-monotone.
\end{lemma}

\begin{proof}
 Let $ A=\{x_0, x_1, \cdots, x_k\} \subset K$ and  $y \in {\rm conv}(A)$ be arbitrary. By Lemma \ref{conv lemma}, there is an integer $n\geq0$ such that $y\in C_n(A)$. We will show $\underset{x\in A}{\min} f(x,y)\leq 0$. Suppose to the contrary, $f(x,y)>\lambda>0,\ \ \forall x\in A$. \\
(i) By quasi-convexity of $-f(\cdot,y)$ and the definition of $C_1(A)$, we have:
$$-f(u,y)\leq \max \{-f(x,y) | x\in A\} <-\lambda<0, \ \ \forall u\in C_1(A)$$
Again by the definition of $C_2(A)$ we get
$$-f(v,y)\leq \max \{-f(u,y) | u\in C_1(A)\} \leq-\lambda<0, \ \ \forall v\in C_2(A)$$
and finally by induction we get $0=-f(y,y)\leq-\lambda<0$, which is a contradiction.\\
ii)
By using $P4^*$, we have $f(y,x_i)\leq 0 $, for  all $i\in\{0,\cdots,k\}$. Now if $f(y,x_i)=0 $ for some $i$,
then again by using $P4^*$, we have $f(x_i,y)\leq 0 $, that is a contradiction.
Therefore $f(y,x_i)< 0 $ for all $0\leq i \leq k$. Next by the quasi-convexity of $f(y,\cdot)$ and Lemma \ref{conv lemma} and similar to the part (i) we get a contradiction, which proves the lemma.
\end{proof}

\begin{theorem}\label{lem23}
Suppose that $f$ is properly quasi-monotone and $P1$, $P2$, $P3$ and $P5$ hold, then $EP(f,K)$ admits a solution.
\end{theorem}

\begin{proof}
Let $n\in \mathbb N$ be arbitrary, we are going to use Lemma \ref{lem kkm} with $K_n$
instead of $K$ and $G(y):=L_f(n,y)$. Therefore we must check the validity of its hypotheses.
First we verify condition (i) of Lemma \ref{lem kkm}.\\
Take $x_0, x_1, \cdots, x_k \in K_n$ and $\bar{x}\in {\rm conv}(\{x_0,x_1,\cdots,x_k\})$.
 We must verify that $\bar{x}\in \overset{k} {\underset{i=0}{\bigcup}}L_f(n,x_i)$, i.e. $\bar{x}\in K_n$
  and $f(x_i,\bar{x})\leq 0$ for some $i$. Since $K_n$ is convex, therefore $\bar{x}\in K_n$ and
  the rest of this fact is followed from properly quasi-monotonicity assumption, which guarantees that $\underset{0\leq i\leq k}{\min} f(x_i,\bar{x})\leq 0$.\\
Now we verify condition (ii) of Lemma \ref{lem kkm}. Since $f(y,\cdot)$ is convex and lower semicontinuous, therefore
$G(y)=L_f(n,y)=\{x\in K_n | f(y,x)\leq0\}$ is closed and convex. Also $G(y)$ is bounded,
because it is contained in $K_n$. Hence by Lemma \ref{compact}, $G(y)$ is $\triangle$-compact,
for all $y\in K$. Therefore we are within the hypotheses of Lemma \ref{lem kkm}
and we can conclude that $\underset{y\in K_n}{\bigcap }L_f(n,y)\neq\varnothing$, for each $n\in \mathbb N$,
so that for each $n\in \mathbb N$ we may choose $x_n\in \underset{y\in K_n}{\bigcap }L_f(n,y)$. We distinguish two cases:\\
i)
There is $n\in \mathbb N$ such that $d(o,x_n)< n$. In this case $x_n \in K_n^\circ$
solves $EP(f,K)$ by Lemma \ref{lem21}.\\
ii)
$d(o,x_n)=n$ for all $n\in \mathbb N$. In this case $P5$ ensure existence of
 $u\in K$ and $n_0>0$ such that $f(x_n,u)\leq0$ for all $n\geq n_0$. Take $n'\geq n_0$ such
 that $d(o,u)<n'$, then $f(x_{n'},u)\leq0$ and $u\in K_{n'}^\circ$.
Again $x_{n'}$ turns out be a solution of $EP(f,K)$ by Lemma \ref{lem21}.
\end{proof}

\begin{theorem}\label{lem25}
Let $f$ satisfy $P1$, $P2$, $P3$ and $P4^*$, then $EP(f,K)$ has a solution if and only if $P5$ holds.
\end{theorem}

\begin{proof}
Take $x^*\in S(f,K)$, then $f(x^*,y)\geq 0$ for all $y\in K$. By $P4^*$ we have $f(y,x^*)\leq 0$ for all $y\in K$. Hence $P5$ holds.\\
Now, by Lemma \ref{lem22}, $P1$ and $P4^*$ imply that $f$ is properly quasi-monotone, then by Theorem \ref{lem23}, $EP(f,K)$ has a solution if $P5$ holds.
\end{proof}

The following theorem also shows the existence of solutions for some equilibrium problems.
It has been essentially proved on finite dimensional Hadamard manifolds in \cite{vitto} and we rewrite the proof in Hadamard spaces.

\begin{theorem}
Let $f:K\times K\rightarrow \mathbb R$ be a bifunction such that\\
i) for any $x\in K$, $f(x,x)\geq0;$\\
ii) for every $x\in K$, the set $\{y\in K| f(x,y)< 0\}$ is convex;\\
iii) for every $y\in K$, $x\mapsto f(x,y)$ is $\triangle$-upper semicontinuous;\\
iv) there exists a $\triangle$-compact set $L\subseteq X$ and a point $y_0\in L\cap K$ such that\\
$f(x,y_0)< 0$, for all $x\in K\backslash L$,\\
then there exists a point $x_0\in L\cap K$  satisfying $f(x_0,y)\geq0$,  for all $y\in K$.
\end{theorem}
\begin{proof}
$G:K\rightarrow 2^K$ is defined by $G(y):=\{x\in K |f(x,y)\geq0\}$, for each $y\in K$. Since $f(\cdot,y)$ is $\triangle$-upper semicontinuous, $G(y)$ is $\triangle$-closed for all $y\in K$. In turn by condition (iv) there exists a point $y_0\in K$  such that $G(y_0)\subseteq L$, so $G(y_0)$ is $\triangle$-compact. We are going to use Lemma \ref{lem kkm}, thus we must prove that for every $y_1, \cdots ,y_m \in K$, we have ${\rm conv}(\{y_1, \cdots ,y_m\})\subset\bigcup_{i=1}^mG(y_i).$\\
To this end, suppose to the contrary that there exists a point $x'$ such that $x' \in {\rm conv}(\{y_1, \cdots ,y_m\})$ but $x'\not\in \bigcup_{i=1}^mG(y_i)$, i.e. $$f(x',y_i)< 0, \  \  \   \ 1\leq i\leq m.$$
This implies that for all $1\leq i\leq m$, we have $y_i \in \{y\in K | f(x',y)< 0\}$. Since $\{y\in K | f(x',y)< 0\}$ is convex hence we have:\\
$x' \in {\rm conv}(\{y_1, \cdots ,y_m\})\subseteq \{y\in K | f(x',y)< 0\}$, but by (i), we have $f(x',x')\geq0$, which is a contradiction.
Then by Lemma \ref{lem kkm}, there exists $x_0\in K$ such that $x_0 \in \bigcap_{y\in K} G(y)$, with $x_0\in G(y_0) \subseteq L\cap K$. In the other words there exists a point $x_0\in L\cap K$  satisfying $f(x_0,y)\geq0$,  for all $y\in K$.
\end{proof}

\section{\bf An Auxiliary Problem}
\noindent
In this section, we consider the proximal point scheme for pseudo-monotone equilibrium problems in Hadamard spaces to
approximate an equilibrium point. The proximal point algorithm for pseudo-monotone bifunction $f:K\times K\rightarrow \mathbb R$ generates the sequence $\{x_k\}$ which is given by the following process. Given $x_0\in X$ arbitrary, inductively for $x_{k-1}\in K$, $x_k$ satisfies in the following inequality
\begin{equation}
f(x_k,y)+\lambda_{k-1} \langle  \overrightarrow{x_{k-1}x_k },\overrightarrow{x_ky }\rangle \geq0, \   \    \ \forall y \in K,
\label{ppa}\end{equation}
where $ \{\lambda_k\}$ is a positive sequence. When $X$ is a Hilbert space and $f$ is $\theta$-undermonotone and  $\lambda_k> \theta, \   \forall k \in \mathbb N$, Iusem and Sosa in \cite{al-wi} proved existence and uniqueness of the sequence generated by (\ref{ppa}). They also proved the weak convergence of the sequence to an equilibrium point of $f$, when $f$ is a pseudo-monotone bifunction. Unfortunately, we cannot obtain existence of the sequence $\{x_k\}$ defined by (\ref{ppa}) in general Hadamard spaces for each bifunction $f$ with the usual conditions $P1$, $P2$, $P3$, $P4$, $P4^*$ and $P4^{\bullet}$ discussed in Section 2. In \cite{kr2} the first author and Ranjbar proved the existence of the sequence defined by (\ref{ppa}) and its $\Delta$-convergence for bifunction $f(x,y)=\langle \overrightarrow{Txx},\overrightarrow{xy}\rangle$, where $T:X\rightarrow X$ is a nonexpansive mapping. In this section, we study the existence of the sequence given by \eqref{ppa} in some other cases. In order to prove existence and uniqueness of the sequence $\{x_k\}$ satisfying \eqref{ppa}, consider the bifunction $\tilde{f}$ which is defined by \begin{equation}
\tilde f(x,y)= f(x,y)+\lambda \langle  \overrightarrow{\bar{x}x },\overrightarrow{xy }\rangle .
\label{ftilde}\end{equation} where $\bar{x}\in X$ and $f$ is a bifunction that satisfies $P1$, $P2$, $P3$ and $P4^{\bullet}$ and $\lambda>\theta$. First we prove the uniqueness of the sequence $\{x_k\}$ satisfying \eqref{ppa}.
Assume that both $x'$ and $x''$ solve $EP(\tilde f, K)$. Note that
$$0\leq \tilde f(x',x'')= f(x',x'')+\lambda \langle  \overrightarrow{\bar{x}x' },\overrightarrow{x'x''}\rangle ,$$
$$0\leq \tilde f(x'',x')= f(x'',x')+\lambda \langle  \overrightarrow{\bar{x}x'' },\overrightarrow{x''x'}\rangle .$$
By summing the both sides of the above inequalities, we have:
$$0\leq f(x',x'')+f(x'',x')-\lambda d^2(x',x'')\leq (\theta-\lambda)d^2(x',x'').$$
Since $\lambda > \theta$, we deduce that $x'=x''$.

\begin{lemma}\label{lem13}
Let $f$ satisfy $P1$, $P2$, $P3$, $P4^{\bullet}$ and $\lambda > \theta$, then $\tilde f$ satisfies $P4$ and $P5$.
\end{lemma}

\begin{proof}First we prove that $\tilde f$ satisfies $P4$. Note that\\
$\tilde f(x,y)+ \tilde f(y,x)= f(x,y)+ f(y,x)+\lambda \langle  \overrightarrow{\bar{x}x },\overrightarrow{xy }\rangle +\lambda \langle  \overrightarrow{\bar{x}y },\overrightarrow{yx }\rangle  \\
=f(x,y)+ f(y,x)-\lambda d^2(x,y) \leq (\theta-\lambda)d^2(x,y) \leq 0$.\\
Now we show that $\tilde f$ satisfies $P5$. Take and fix $o \in X$, then take a sequence $\{x_k\}$ such that $\lim d(o,x_k)=+\infty$ and let $u=P_K(\bar{x})$, where $P_K :X\longrightarrow K$ is the projection map onto $K$. Since by Theorem 2.2 of \cite{dehghan}, we have, $\langle  \overrightarrow{\bar{x}u },\overrightarrow{x_ku }\rangle \leq0$. Therefore we have:\\
$\tilde f(x_k,u)= f(x_k,u)+\lambda \langle  \overrightarrow{\bar{x}x_k },\overrightarrow{x_ku }\rangle = f(x_k,u)+\lambda \langle  \overrightarrow{\bar{x}u },\overrightarrow{x_ku }\rangle +\lambda \langle  \overrightarrow{ux_k },\overrightarrow{x_ku }\rangle  \\
\leq f(x_k,u)-\lambda d^2(u,x_k) \leq -f(u,x_k)+\theta d^2(u,x_k)-\lambda d^2(u,x_k)$
\begin{equation}
= -f(u,x_k)-(\lambda-\theta)d^2(u,x_k),
\label{q5}\end{equation}
where in the second inequality we have used from $\theta$-undermonotonicity of $f$.
Now take $z$ in the domain of $f(u,\cdot)$ and $t \in \mathbb{R}$ with $t< f(u,z)$, since $f(u,\cdot)$ is convex, proper and lower semicontinuous, by Lemma 3.2 of \cite{aka} there are $v\in X$ and a real number $t< s \leq f(u,z)$ such that
$$f(u,y)\geq \frac{1}{s-t}\langle  \overrightarrow{vz },\overrightarrow{vy }\rangle +s \   \    \     \ \forall y\in K,$$
therefore by setting $y=x_k$, we have: $-f(u,x_k)\leq \frac{-1}{s-t}\langle  \overrightarrow{vz },\overrightarrow{vx_k }\rangle -s$. Now by Cauchy- Schwarz inequality, we have:
\begin{equation}
-f(u,x_k)\leq \frac{1}{s-t}d(v,z)d(v,x_k)-s \leq \frac{1}{s-t}d(v,z)d(u,x_k)+\frac{1}{s-t}d(v,z)d(v,u)-s.
\label{q5'}\end{equation}
By replacing (\ref{q5'}) in (\ref{q5}), we have:
\begin{equation}
\tilde f(x_k,u)\leq d(x_k,u)[\frac{1}{s-t}d(v,z)-(\lambda-\theta)d(u,x_k)]+ \frac{1}{s-t}d(v,z)d(v,u)-s.
\label{q5''}\end{equation}
Since $\lambda-\theta > 0$ and $\lim d(x_k,o)=+\infty$, so that $\lim d(x_k,u)=+\infty$,
it follows easily from (\ref{q5''}) that $\lim \tilde f(x_k,u)=-\infty$ as $k\rightarrow +\infty$. So that $\tilde f(x_k,u)\leq 0$, for sufficiently large $k$. Therefore $\tilde f$ satisfies $P5$.
\end{proof}

\begin{proposition}\label{convex}
Let $f$ satisfy $P1$, $P2$, $P3$, $P4^{\bullet}$ and $\lambda > \theta$. If $\tilde f(x,\cdot)$ is convex for all $x\in K$, then $EP(\tilde f,K)$ has a unique solution.
\end{proposition}

\begin{proof}
It is clear that $\tilde f$ satisfies $P1$, $P2$ and $P3$. Also Lemma \ref{lem13} shows that $\tilde f$ satisfies $P4$ and $P5$. Hence Theorem  \ref{lem25} implies that $\tilde f$ has a solution. Uniqueness of the solution has already been proved.
\end{proof}

By Proposition \ref{convex}, if $ \tilde f(x,\cdot)$ is convex for all
$x\in K$, then $EP( \tilde f,K)$ has a unique solution, but since $y\mapsto\langle \overrightarrow{\bar{x}x},\overrightarrow{xy}\rangle $ is not convex in general unless in flat Hadamard spaces, then the conditions of Proposition \ref{convex} are satisfied in these spaces and we have the following corollary.

\begin{corollary}
	Let $f$ satisfy $P1$, $P2$, $P3$, $P4^{\bullet}$ and $\lambda > \theta$. If $X$ is a flat Hadamard space, then $EP(\tilde f,K)$ has a unique solution.
\end{corollary}

If the function $y\mapsto\langle \overrightarrow{\bar{x}x},\overrightarrow{xy}\rangle $ is convex, existence of a solution for $\tilde{f}$ is concluded by the usual conditions on the bifunction $f$. But in general $y\mapsto\langle \overrightarrow{\bar{x}x},\overrightarrow{xy}\rangle $ is not convex in Hadamard spaces. In the following theorems we try to overcome this problem and prove the existence of solutions for $\tilde{f}$ in some special  cases.

In order to prove existence of an equilibrium point for $\tilde{f}$ when $f$ is cyclic monotone we recall the definition of cyclic monotonicity of bifunctions  from \cite{had-kha} and a lemma that we need to prove the main result.

\begin{definition}
	$f:K\times K\rightarrow\mathbb{R}$ is said to be cyclic monotone iff for each $n\in\mathbb{N}$ and each $x_1,x_2,\cdots, x_n\in X$ $$f(x_1,x_2)+f(x_2,x_3)+\cdots +f(x_n,x_1)\leq0$$
\end{definition}

\begin{lemma}
	Suppose that $f:K\times K\rightarrow\mathbb{R}$ is monotone and $P1$ is satisfied. Also $f$ is convex respect to the second variable and upper hemi-continuous (upper semi-continuous along geodesics) respect to the first variable. Let $\bar{x}\in K$, then the following are equivalent:\\
	i) there exists $x\in K$ such that $f(z,x)+\langle \overrightarrow{x\bar{x}},\overrightarrow{xz}\rangle \leq0,\ \ \forall z\in K$\\
	ii) there exists $x\in K$ such that $f(x,z)\geq\langle \overrightarrow{x\bar{x}},\overrightarrow{xz}\rangle ,\ \ \forall z\in K$
\end{lemma}

\begin{proof} (ii) $\Rightarrow$ (i) is trivial by the monotonicity of $f$. We prove (i) $\Rightarrow$ (ii).
	For all $z\in K$ and $0\leq t\leq1$, take $z_t=tz\oplus(1-t)x$. By convexity of $f$ respect to the second argument, we have\\
	$0=f(z_t,z_t)\leq t f(z_t,z)+(1-t)f(z_t,x)\leq tf(z_t,z)+(1-t)\langle \overrightarrow{\bar{x}x},\overrightarrow{xz_t}\rangle $\\
	$=tf(z_t,z)+\frac{1-t}{2}\{d^2(\bar{x},z_t)-d^2(x,z_t)-d^2(\bar{x},x)\}$\\
	$=tf(z_t,z)+\frac{1-t}{2}\{td^2(\bar{x},z)+(1-t)d^2(\bar{x},x)-t(1-t)d^2(z,x)-t^2d^2(z,x)-d^2(\bar{x},x)\}$\\
	$=tf(z_t,z)+\frac{t(1-t)}{2}\{d^2(\bar{x},z)-d^2(\bar{x},x)-d^2(z,x)\}$\\
	Therefore
	$$f(z_t,z)\geq(1-t)\langle \overrightarrow{x\bar{x}},\overrightarrow{xz}\rangle $$
	Letting $t\rightarrow0$, by upper hemi-continuity of $f$ respect to the first argument, we get
	$$f(x,z)\geq\langle \overrightarrow{x\bar{x}},\overrightarrow{xz}\rangle, \ \ \forall z\in K$$
	as desired.
\end{proof}

\begin{theorem}
	Let $f:K\times K\rightarrow\mathbb{R}$ be a cyclic monotone bifunction which satisfy $P1$, $P2$ and $P3$. Then $\tilde{f}$ has a solution.
\end{theorem}

\begin{proof}
	Without loss generality from now to the end of this section, we take $\lambda=1$ in $EP(\tilde{f},K)$. By similar argument of Propositions 5.1 and 5.2 of \cite{had-kha}, $f(z,x)\leq g(x)-g(z)$, where $g$ is a convex and lower semi-continuous function on $X$. By Theorem 4.2 of \cite{aka}, for a given $\bar{x}\in K$ there exists exactly one $x\in K$ such that
	$$g(x)-g(z)\leq\langle \overrightarrow{\bar{x}x},\overrightarrow{xz}\rangle ,\ \ \forall z\in K,$$
	then
	$$f(z,x)+\langle \overrightarrow{x\bar{x}},\overrightarrow{xz}\rangle \leq0,\ \ \forall z\in K$$
	Now Lemma 3.6 implies the interested result.
\end{proof}

Now we want to prove existence of an equilibrium point for $\tilde{f}$ when $f$ satisfies a cyclic pseudo-monotonicity condition. In \cite{had-sch-won} cyclic pseudomonotonicity was defined for pseudomonotone operators. In \cite{kha-moh-ali} we defined it for pseudomonotone bifunctions as follows.\\
i)  $f$\ is called $n\,$-pseudomonotone if the following implication holds:
\[
f\left(  x_{1},x_{2}\right)  \geq0,f\left(  x_{2},x_{3}\right)  \geq
0,...,f\left(  x_{n-1},x_{n}\right)  \geq0\Longrightarrow f\left(  x_{n}
,x_{1}\right)  \leq0;
\]
ii) $f$ is called  cyclic pseudomonotone, if $f$ is $n$-pseudomonotone for all
$n\in
\mathbb{N}
$.\\
In order to prove existence of solution for $\tilde{f}$, we define a stronger version of cyclic pseudo-monotonicity as follows.

\begin{definition}
	$f$\ is called $n\,$-pseudomonotone of type (I), if the following
 implication holds:
 \[
 f\left(  x_{1},x_{2}\right)  +f\left(  x_{2},x_{3}\right)  +...+f\left(
 x_{n-2},x_{n-1}\right)  \leq f\left(  x_{1},x_{n}\right)  \Longrightarrow
 f\left(  x_{n},x_{n-1}\right)  \leq0
 \]
 \end{definition}

 First we prove that the recent definition is stronger than the definition of cyclic pseudomonotonicity.

 \begin{theorem}\label{n-pseudomonotone}
 	If $f:K\times K\rightarrow\mathbb{R}$ is n-pseudomonotone of type (I) and $P_1$ is satisfied, then $f$ is n-pseudomonotone.
 \end{theorem}

 \begin{proof}
 	Take $x_1,x_2,...,x_n\in K$, and let $f(x_1,x_2)\geq0$, $f(x_2,x_3)\geq0$, ... , $f(x_{n-1},x_n)\geq0$. Since $P1$ and n-pseudomonotonicity of type (I) imply that $P_4^*$, hence we have:\\
 	$f(x_{n-1},x_{n-2})+ f(x_{n-2},x_{n-3})+... + f(x_2,x_1)\leq f(x_{n-1},x_n)$.\\
 	Now n-pseudomonotonicity of type (I) implies $f(x_n,x_1)\leq0$, as desired.
 \end{proof}

\begin{proposition}\label{gx-gy}
	Assume that $f:K\times K\rightarrow
	\mathbb{R}
	$ is cyclic pseudomonotone of type (I), and there are $u,v\in  K$ such that $f(u,v)> 0$. Then\\
	(i) there exists a function $g:K\rightarrow
	\mathbb{R}
	$ such that $f\left(  x,y\right)  \geq g\left(  y\right)-g\left(  x\right)
	;$\\
	(ii) if $f\left(  \cdot,y\right)  $ is concave for each $y\in K,$ then $g$ is convex.
\end{proposition}

\begin{proof}
	(i) Let $f:K\times K\rightarrow\mathbb{R}$ be n-pseudomonotone of type (I) and $f(u,v)> 0$. Note that the definition of n-pseudomonotone of type (I) implies
	$f(x_1,x_2)+f(x_2,x_3)+...+f(x_{n-2},v)> f(x_1,u)$ for each $x_1,x_2,\cdots, x_{n-2}\in K$. Now we define
	$\varphi(x_1)=\inf_{x_2,...,x_{n-2},n\geq3}\{f(x_1,x_2)+f(x_2,x_3)+...+f(x_{n-2},v)\}\geq f(x_1,u)$. Hence we have $f(x_1,x_2)+\inf_{x_3,...,x_{n-2},n\geq3}\{f(x_2,x_3)+...+f(x_{n-2},v)\}\geq \inf_{x_2,...,x_{n-2},n\geq3}\{f(x_1,x_2)+...+f(x_{n-2},v)\}$, therefore $f(x_1,x_2)\geq \varphi(x_1)-\varphi(x_2)$. Taking $g=-\varphi$, we have $f\left(  x,y\right)  \geq g\left(  y\right)-g\left(  x\right)$.

	(ii) If $\forall y\in K, f(\cdot,y)$ is concave, then for all $\lambda\in(0,1)$ and $z_1, z_2 \in K$, we have: \\
	$\varphi(\lambda z_1+(1-\lambda)z_2)=
	\inf_{x_2,...,x_{n-2},n\geq3}\{f(\lambda z_1+(1-\lambda)z_2,x_2)+...+f(x_{n-2},v)\}\geq \lambda\inf_{x_2,...,x_{n-2},n\geq3}\{f( z_1,x_2)+...+f(x_{n-2},v)\}+
	(1-\lambda)\inf_{x_2,...,x_{n-2},n\geq3}\{f(z_2,x_2)+...+f(x_{n-2},v)\}
	=\lambda \varphi(z_1)+(1-\lambda)\varphi(z_2)$. Therefore $g=-\varphi$ is convex.
\end{proof}

\begin{theorem}
		Let $f$ be a cyclic pseudomonotone bifunction of type (I) which satisfies P2. If $f$ is concave respect to the first argument and there exist $u,v\in K$ such that $f(u,v)> 0$, then $\tilde{f}$ has a solution.
\end{theorem}

\begin{proof}
	By Proposition \ref{gx-gy}, $f(x,y)\geq g(y)-g(x)$, where $g$ is convex and lower semicontinuous. Therefore by Theorem 4.2 of \cite{aka}, for each $\bar{x}\in K$ there exists $x\in K$ such that  $$g(y)-g(x)\geq\langle \overrightarrow{x\bar{x}},\overrightarrow{xy}\rangle ,\ \ \forall y\in K.$$ This implies that $\tilde{f}$ has a solution.
\end{proof}

\begin{problem}
	 We have already proved existence of solution for $EP(\tilde{f},K)$ or equivalently existence of a sequence which satisfies \eqref{ppa} by imposing some conditions on the monotone or pseudomonotone  bifunction $f$ and the Hadamard space $X$, but we don't know whether the problem \eqref{ftilde} has a solution without these extra conditions.
\end{problem}

\section{ \bf Convergence of Resolvent}
\noindent
Now consider a monotone bifunction $f:K\times K\rightarrow\mathbb{R}$. Assume that for each $\lambda> 0$ and $\bar{x}\in K$, the equilibrium problem for $\tilde{f}$ (see \eqref{ftilde}) has a solution that is unique. This unique solution is denoted by $J^f_{\lambda}\bar{x}$ and it is called the resolvent of $f$ of order $\lambda> 0$ at $\bar{x}$. The resolvent $J^f_{\lambda}$ or briefly $J_{\lambda}$ for monotone bifunctions in Hilbert and Banach spaces has been introduced by Ait Mansour et al. in \cite{ait}(see also \cite{had-kha}). In Hadamard spaces, we proved existence of the resolvent in some special cases in the previous section. In the following theorem we prove $J_{\lambda}$ is firmly nonexpansive and then prove that for each $x\in X$, $J_{\lambda}x$ converges strongly to an equilibrium point of $f$ as $\lambda\rightarrow0$, if $S(f,K)\neq\varnothing$. First we recall the definitions of firmly nonexpansive and quasi firmly nonexpansive mappings.

\begin{definition}
	A mapping $T:X\rightarrow X$ is called firmly nonexpansive iff $$\langle \overrightarrow{xy},\overrightarrow{TxTy}\rangle \geq d^2(Tx,Ty),\ \ \forall x,y\in X$$
	$T$ is called quasi firmly nonexpansive if ${\rm Fix}(T)\neq\varnothing$,
where ${\rm Fix}(T)$ is the set of all fixed point of $T$, and $$\langle \overrightarrow{xp},\overrightarrow{Txp}\rangle \geq d^2(Tx,p),\ \ \forall x\in X$$ for each $p\in {\rm Fix}(T)$.
\end{definition}

\begin{proposition}\label{firmly nonexpansive}
	Let $f:K\times K\rightarrow\mathbb{R}$ be a bifunction and $\lambda> 0$ such that $J_{\lambda}x$ exists.\\
	i) If $f$ is monotone, then the mapping $x\mapsto J_{\lambda}x$ is firmly nonexpansive \\
	ii) If $f$ is pseudomonotone and $S(f,K)\neq\varnothing$, then $J_{\lambda}x$ is quasi-firmly nonexpansive.
\end{proposition}

\begin{proof}
	(i) First suppose that $f$ is monotone. Take two points $x,z\in X$. We have
\begin{equation}f(J_{\lambda}x,y)+\lambda\langle \overrightarrow{xJ_{\lambda}x},\overrightarrow{J_{\lambda}xy}\rangle \geq0,\ \ \ \forall y\in K\label{resolvent ppa x}\end{equation}
and also
\begin{equation}f(J_{\lambda}z,y)+\lambda\langle \overrightarrow{zJ_{\lambda}z},\overrightarrow{J_{\lambda}zy}\rangle \geq0,\ \ \ \forall y\in K\label{resolvent ppa z}\end{equation}
Now letting $y=J_{\lambda}z$ in \eqref{resolvent ppa x} and $y=J_{\lambda}x$ in \eqref{resolvent ppa z}.
Then summing the recent inequalities, by the monotonicity of $f$, we get
$$\langle \overrightarrow{xJ_{\lambda}x},\overrightarrow{J_{\lambda}xJ_{\lambda}z}\rangle +\langle \overrightarrow{zJ_{\lambda}z},\overrightarrow{J_{\lambda}zJ_{\lambda}x}\rangle \geq0$$By a straightforward computation and using quasi-inner product properties, we get
$$\langle \overrightarrow{xz},\overrightarrow{J_{\lambda}xJ_{\lambda}z}\rangle \geq d^2(J_{\lambda}x,J_{\lambda}z),$$
which follows the desired result. Also the recent inequality implies nonexpansiveness of $J_{\lambda}$ by Cauchy-Schwarz inequality.\\
(ii) If $f$ is pseudomonotone, then set $y=p\in S(f,K)$ in \eqref{resolvent ppa x}, since $f(J_{\lambda}x,p)\leq0$, we get
$$\langle \overrightarrow{xJ_{\lambda}x},\overrightarrow{J_{\lambda}xp}\rangle \geq0,$$
which implies that
$$\langle \overrightarrow{xp},\overrightarrow{J_{\lambda}xp}\rangle \geq d^2(J_{\lambda}x,p).$$

\end{proof}

It is easy to see that in two cases of Proposition \ref{firmly nonexpansive}, $S(f,K)={\rm Fix}(J_{\lambda})$.

Before the main result of this section we need to prove  Kadec-Klee property in Hadamard spaces.
\begin{proposition}\label{kadek kelly}
	Suppose that $x_n$ is $\Delta$-convergent to $x$ and there exists $y\in X$ such that $\limsup d(x_n,y)\leq d(x,y)$, then $x_n$converges strongly to $x$.
\end{proposition}

\begin{proof}
	By the definition and properties of quasi-linearization, we have\\
	$d^2(x_n,x)=\langle \overrightarrow{x_nx},\overrightarrow{x_nx}\rangle =\langle \overrightarrow{x_nx},\overrightarrow{x_ny}\rangle +\langle \overrightarrow{x_nx},\overrightarrow{yx}\rangle $\\$=\langle \overrightarrow{x_ny},\overrightarrow{x_ny}\rangle +\langle \overrightarrow{yx},\overrightarrow{x_ny}\rangle +\langle \overrightarrow{x_nx},\overrightarrow{yx}\rangle $\\
	$=d^2(x_n,y)+2\langle \overrightarrow{yx},\overrightarrow{x_nx}\rangle -d^2(x,y)$\\
	Taking limsup when $n\rightarrow+\infty$, by Theorem 2.6 of \cite{ak} and the hypotheses, we get
	$$\limsup d^2(x_n,x)\leq\limsup d^2(x_n,y)-d^2(x,y)\leq0$$
	as desired.
\end{proof}

\begin{theorem} \label{convergence resolvent}
	Let $f:K\times K\rightarrow\mathbb{R}$ be a monotone bifunction that satisfy $P1, P3$ and $\Delta$-upper semicontinuity respect to the first argument and $S(f,K)\neq\varnothing$. If for each $\lambda> 0$ and $x\in K$, $J_{\lambda}x$ exists, then $J_{\lambda}x$ converges strongly to $p\in S(f,K)$, which is the nearest point of $S(f,K)$ to $x$ as $\lambda\rightarrow0$.
\end{theorem}

\begin{proof}
	Take $p\in S(f,K)$. By Proposition \ref{firmly nonexpansive}, $d(J_{\lambda}x,p)\leq d(x,p)$. Therefore $\{J_{\lambda}x\}$ is bounded.
	 Suppose that there is a sequence $\lambda_n$ converges to $0$ such that $J_{\lambda_n}x$ $\Delta$-converges to $q$. By $\Delta$-upper semicontinuity of $f$ and \eqref{resolvent ppa x}, we get $f(q,y)\geq0$ for each $y\in K$, hence $q\in S(f,K)$. Note that by monotonicity of $f$, $f(J_{\lambda}x,p)\leq0$, for all $p\in S(f,K)$. Therefore (4.1) implies that $\langle \overrightarrow{xJ_{\lambda}x},\overrightarrow{J_{\lambda}xp}\rangle \geq0$. Hence $$d^2(J_{\lambda}x,x)\leq\langle \overrightarrow{xJ_{\lambda}x},\overrightarrow{xp}\rangle ,\ \ \forall p\in S(f,K)$$
	By Cauchy-Schwarz inequality
	\begin{equation}d(x,J_{\lambda}x)\leq d(x,p),\ \ \ \ \forall p\in S(f,K)\label{resolvent p}\end{equation}
	 Now taking $\lambda=\lambda_n$ and taking liminf when $n\rightarrow+\infty$, since $d(x,\cdot)$ is convex and continuous, then it is $\Delta$-lower semicontinuous and we get
	$$d(x,q)\leq d(x,p),\ \ \ \ \forall p\in S(f,K)$$ Therefore $q=P_{S(f,K)}x$. This proves the $\Delta$-convergence of $J_{\lambda}x$ to $P_{S(f,K)}x$ as $\lambda\rightarrow0$. By \eqref{resolvent p}, we have
	$$d(x,J_{\lambda}x)\leq d(x,P_{S(f,K)}x)$$
	Now by Proposition \ref{kadek kelly}, $J_{\lambda}x$ converges strongly to $P_{S(f,K)}x$ as $\lambda\rightarrow0$.
\end{proof}
	
	 \begin{remark}
	 	Theorem \ref{convergence resolvent} is true also for pseudomonotone bifunctions if $J_{\lambda}x$ exists. But since by Proposition \ref{convex}, condition $\lambda> \theta\geq 0$ is essential for existence of solution to $\tilde{f}$ and therefore existence of $J_{\lambda}x$, we explored Theorem \ref{convergence resolvent} only for monotone bifunctions.
	 \end{remark}

\section{\bf Proximal Point Algorithm}
\noindent
In this section, we study the convergence of proximal point method
for equilibrium problems that the bifunction $f$ satisfies $P1$,
$P2$, $P3$, $P4^*$ and $P4^{\bullet}$ by assuming existence
of a sequence that satisfies \eqref{ppa}. For computational and
numerical purposes and since the existence of the sequence satisfying \eqref{ppa} is not guaranteed in general, we consider an inexact version of \eqref{ppa}.
Let $\theta$ be the under monotonicity constant of $f$. Take a
sequence of regularization parameters $\{\lambda_k\}\subset
(\theta, \bar{\lambda}]$,
for some $\bar{\lambda} > \theta$. Take $x_0 \in X$ and construct the sequence $\{x_k\}\subset K$ as follows:\\
Given $x_k$, then we take $y_k$ such that $d(x_k,y_k)\leq e_k$, in
turn $x_{k+1}$ is a unique solution of problem
$EP(f_k,K)$, where $f_k:K\times K\rightarrow\mathbb{R}$ is
defined as
\begin{equation}
f_k(x,y)= f(x,y)+\lambda_k \langle  \overrightarrow{y_kx
},\overrightarrow{xy }\rangle.
\label{regularization}\end{equation} where
$\sum_{k=1}^{\infty}e_k<+\infty$. Throughout this section, we assume that $\sum_{k=1}^{\infty}e_k< +\infty$.

\begin{lemma}\label{bounded}
Consider $EP(f,K)$, where $f$ satisfies $P1$, $P2$, $P3$, $P4^*$ and $P4^{\bullet}$. If $EP(f,K)$ has a solution, then the sequence $\{x_k\}$, which is generated by (\ref{regularization}),  is bounded and $\lim d(x_k,x_{k+1})=0$.\\
\end{lemma}

\begin{proof}
Take $x^* \in S(f,K)$. Note that $f(x_{k+1},x^*)+\lambda_k \langle  \overrightarrow{y_kx_{k+1} },\overrightarrow{x_{k+1}x^* }\rangle \geq0$. Since $f(x_{k+1},x^*)\leq 0$, hence we have $ \langle  \overrightarrow{y_kx_{k+1} },\overrightarrow{x_{k+1}x^* }\rangle \geq0$, which implies that
$$d^2(x_{k+1},x^*)+ d^2(y_k,x_{k+1})\leq d^2(y_k ,x^*).$$ Therefore, we conclude\\
$d^2(x_{k+1},x^*)+ d^2(y_k,x^*)+d^2(x_{k+1},x^*)-2 \langle  \overrightarrow{y_kx^*},\overrightarrow{x_{k+1}x^*}\rangle \leq d^2(y_k ,x^*)$\\
$\Longrightarrow 2d^2(x_{k+1},x^*)\leq 2 \langle  \overrightarrow{y_kx^*},\overrightarrow{x_{k+1}x^*}\rangle \leq 2d(y_k ,x^*)d(x_{k+1},x^*)$\\
$\Longrightarrow d(x_{k+1},x^*)\leq d(y_k ,x^*)\leq d(x_k,x^*)+d(x_k,y_k)$\\
Hence, we have
\begin{equation}
d(x_{k+1},x^*)\leq d(x_k,x^*)+e_k.
\label{lim-exist}\end{equation}
Therefore $\lim d(x_k,x^*)$ exists.\\
Also from $d^2(x_{k+1},x^*)+ d^2(y_k,x_{k+1})\leq d^2(y_k ,x^*)$, we have: \\
$d^2(x_{k+1},x^*)+d^2(x_k,y_k)+d^2(x_{k+1},x_k)-2\langle  \overrightarrow{y_kx_k},\overrightarrow{x_{k+1}x_k}\rangle  \\
\leq d^2(x_k,y_k)+d^2(x_k ,x^*)+2\langle  \overrightarrow{y_kx_k},\overrightarrow{x_kx^*}\rangle .$\\
Thus, we have:\\
$d^2(x_{k+1},x_k)\leq d^2(x_k ,x^*)-d^2(x_{k+1},x^*)+2\langle  \overrightarrow{y_kx_k},\overrightarrow{x_{k+1}x^*}\rangle .$\\
By Cauchy-Schwarz inequality, we get:\\
$d^2(x_{k+1},x_k)\leq d^2(x_k ,x^*)-d^2(x_{k+1},x^*)+2e_kd(x_{k+1},x^*)$.\\
Since $\lim d(x_k,x^*)$ exists and $\sum_{k=1}^{\infty}e_k<+\infty$,
therefore $\lim d(x_{k+1},x_k)=0$.\\
\end{proof}

\begin{theorem}\label{delta convergence ppa}
Consider $EP(f,K)$, where $f$ satisfies $P1$, $P3$, $P4^*$ and  $P4^{\bullet}$.
If $f(\cdot,y)$ is $\triangle$-upper semicontinuous for all $y\in K$ and $EP(f,K)$ has a solution, then the sequence $\{x_k\}$ generated by (\ref{regularization}), is $\triangle$-convergent to some solution of $EP(f,K)$.\\
\end{theorem}
\begin{proof}
Fix $y\in K$. Since $ x_{k+1}$ solves $EP(f_k,K)$, hence we have:\\
$0\leq f_k(x_{k+1},y)= f(x_{k+1},y)+\lambda_k \langle  \overrightarrow{y_kx_{k+1} },\overrightarrow{x_{k+1}y }\rangle $
$\leq f(x_{k+1},y)+\lambda_k d(y_k,x_{k+1})d(x_{k+1},y )$\\
$\leq f(x_{k+1},y)+\lambda_k [d(y_k,x_k)+d(x_k,x_{k+1})]d(x_{k+1},y )$.\\
Since $\{\lambda_k\}$ and $\{x_k\}$ are bounded and by Lemma \ref{bounded}, $\lim d(x_{k+1},x_k)=0$, we have:
\begin{equation}
0\leq \liminf f(x_k,y), \   \   \  \forall y\in K.
\label{liminff}\end{equation}
On the other hand, since $\{x_k\}$ is bounded and $K$ is closed and convex, there exist a subsequence $\{x_{k_i}\}$ of $\{x_k\}$ and $x'\in K$ such that $x_{k_i}\overset{\triangle}{\longrightarrow}x'$. Now since $f(\cdot,y)$ is $\triangle$-upper semicontinuous for all $y\in K$, we have:
$$0\leq \liminf f(x_{k},y)\leq \limsup f(x_{k_i},y)\leq f(x',y)$$
for all $y\in K$. So that $x'\in S(f,K)$.\\
It remains to prove that there exists only one $\triangle$-cluster point of $\{x_k\}$. Let $x',x''$ be two $\triangle$-cluster points of $\{x_k\}$ so that there exist two subsequences $\{x_{k_i}\}$ and $\{x_{k_j}\}$ of $\{x_k\}$ whose $\triangle-\lim$  points are $x'$ and $x''$ respectively. We have already proved that $x'$ and $x''$
are solutions of $EP(f,K)$.
In turn by (\ref{lim-exist}), we can assume that $\lim d(x_k,x')=\delta_1$ and $\lim d(x_k,x'')=\delta_2$.
On the other hand, we have:
$$2\langle  \overrightarrow{x_{k_i}x_{k_j}},\overrightarrow{x''x'}\rangle = d^2(x_{k_i},x')+d^2(x_{k_j},x'')-d^2(x_{k_i},x'')-d^2(x_{k_j},x').$$
Letting $i\rightarrow +\infty$, and then $j\rightarrow +\infty$, we get $\underset{j\rightarrow +\infty}{\lim}\underset{i\rightarrow +\infty}{\lim} \langle  \overrightarrow{x_{k_i}x_{k_j}},\overrightarrow{x''x'}\rangle =0$.
Also, we can write the left side of the above statement as:
$$2\langle  \overrightarrow{x_{k_i}x_{k_j}},\overrightarrow{x''x'}\rangle =
2\langle  \overrightarrow{x_{k_i}x'},\overrightarrow{x''x'}\rangle +
2\langle  \overrightarrow{x'x''},\overrightarrow{x''x'}\rangle +
2\langle  \overrightarrow{x''x_{k_j}},\overrightarrow{x''x'}\rangle .$$
By taking $\limsup$ from the above statement and using Theorem 2.6 of \cite{ak}, we conclude, $d^2(x',x'')\leq0$, hence $x'=x''$.
This establishes  that the set of all $\triangle$-cluster points of $\{x_k\}$ is singleton.
\end{proof}

\begin{definition}
A bifunction $f:K\times K\rightarrow\mathbb{R}$ is called strongly monotone, if there exists $\alpha > 0$ such that:
$f(x,y)+f(y,x)\leq -\alpha d^2(x,y)$, for all $x,y \in K$.\\
Also, a bifunction $f:K\times K\rightarrow\mathbb{R}$ is called strongly pseudo-monotone, if there exists a $\beta> 0$ such that if $f(x,y)\geq 0$, then $f(y,x)\leq -\beta d^2(x,y)$,  for all $x,y \in K$.
\end{definition}

It is obvious that if $f$ is strongly monotone, then  $f$ is strongly pseudo-monotone.

\begin{theorem}
Consider $EP(f,K)$, where $f$ satisfies $P1$, $P2$, $P3$, $P4^*$, $P4^{\bullet}$ and $S(f,K)\neq\varnothing$.
If each one of the following conditions satisfies:\\
i) $f$ is strongly pseudo-monotone,\\
ii) $f(x,\cdot)$ is strongly convex for all $x \in K$,\\
iii) $f(\cdot,y)$ is strongly concave for all $y \in K$,\\
then the sequence $\{x_k\}$  generated by (\ref{regularization}), is strongly convergent to a point of $S(f,K)$.
\end{theorem}

\begin{proof}
Take $x^* \in S(f,K)$. In each part, we show $x_k$ converges strongly to $x^* \in S(f,K)$.\\
i)
Since $f(x^*,x_k)\geq0$, by assumption there is $\beta > 0$ such that, $f(x_k,x^*)\leq -\beta d^2(x_k,x^*)$  for all $k\in \mathbb N$.
Next, by (\ref{liminff}) in the proof of Theorem \ref{delta convergence ppa}, we have $\liminf f(x_k,x^*)\geq0$. Therefore by taking liminf, we have:\\
$0\leq \liminf f(x_k,x^*)\leq \liminf(-\beta d^2(x_k,x^*))=-\beta\limsup d^2(x_k,x^*)$
and hence, we deduce that  $x_k\longrightarrow x^*$.\\
ii)
Let $\lambda \in (0,1)$ and set $w_k=\lambda x_k \oplus (1-\lambda)x^*$ for all $k\in \mathbb N$. Since $f(x_k,\cdot)$ is strongly convex, we have:\\
$0\leq f(x_k,w_k)+\lambda_{k-1}\langle \overrightarrow{y_{k-1}x_k},\overrightarrow{x_kw_k}\rangle \\
\leq\lambda f(x_k,x_k) + (1-\lambda) f(x_k,x^*)-\lambda(1-\lambda)
d^2(x_k,x^*)+(1-\lambda)\lambda_{k-1}\langle \overrightarrow{y_{k-1}x_k},\overrightarrow{x_kx^*}\rangle $.\\
Hence, we have:
$ \lambda d^2(x_k,x^*) \leq \lambda_{k-1}\langle \overrightarrow{y_{k-1}x_k},\overrightarrow{x_kx^*}\rangle $.
By using Cauchy-Schwarz inequality, we get: $ \lambda d(x_k,x^*)\leq \lambda_{k-1}d(y_{k-1},x_k)\leq \lambda_{k-1}(d(y_{k-1},x_{k-1})+d(x_{k-1},x_k))\leq \lambda_{k-1} (e_{k-1}+d(x_{k-1},x_k))$. Now, from Lemma  \ref{bounded}, we conclude that $x_k\longrightarrow x^*$.

iii)
Let $\lambda \in (0,1)$ and set $w_k=\lambda x_k \oplus (1-\lambda)x^*$, for all $k\in \mathbb N$. Since $f(\cdot,x^*)$ is strongly concave, we have:\\
$ \lambda f(x_k,x^*) + (1-\lambda) f(x^*,x^*)+\lambda(1-\lambda)
d^2(x_k,x^*)\leq f(w_k,x^*)\leq 0$.\\
Now we get: $ f(x_k,x^*) \leq -(1-\lambda)d^2(x_k,x^*)$. Next, by (\ref{liminff}) in the proof of Theorem \ref{delta convergence ppa} and taking liminf, we have:\\
$0\leq \liminf f(x_k,x^*)\leq -(1-\lambda)\limsup d^2(x_k,x^*)$,
and hence, we deduce that the sequence $\{x_k\}$ is strongly convergent to $x^* \in S(f,K)$.

\end{proof}

\section{\bf Halpern Regularization}
Let $K\subseteq X$ be closed and convex, and $f:K\times K\rightarrow\mathbb{R}$  be a bifunction and
suppose that $\theta$ is the undermonotonicity constant of $f$. Take a
sequence of regularization parameters $\{\lambda_k\}\subset
(\theta, \bar{\lambda}]$,
for some $\bar{\lambda} > \theta$ and $x_0 \in X$.
Consider the following Halpern regularization of the proximal point algorithm for equilibrium problem:
\begin{equation}\begin{cases}f(y_k,y)+\lambda_{k-1}\langle \overrightarrow{x_{k-1}y_k},\overrightarrow{y_ky}\rangle \geq 0 ,\ \ \ \ \   \forall y\in K,\\
x_k=\alpha_k u\oplus(1-\alpha_k)y_k,\\
\end{cases}\label{bi-hc}\end{equation}
where $u \in X$ and the sequence $\{\alpha_k\}\subset(0,1)$ satisfies $\lim \alpha_k=0$ and $\sum_{k=1}^{+\infty}\alpha_k= +\infty$. We will prove the strong convergence of the generated sequence by (\ref{bi-hc}) to a solution of $EP(f,K)$ that the bifunction $f$ satisfies $P1$,
$P2$, $P3$, $P4^*$ and $P4^{\bullet}$. In fact, we prove $x_k\rightarrow x^*={\rm Proj}_{S(f,K)}u$.
First we give an essential lemma that we need in the sequel.

\begin{lemma}\cite{ss1} \label{hss1}
	Let $\lbrace s_{k}\rbrace$ be a sequence of nonnegative real numbers, $\lbrace a_{k}\rbrace$ be a sequence of real numbers in $(0, 1)$ with $\sum _{k=1}^{\infty} a_{k} = +\infty$ and $\lbrace t_{k}\rbrace$ be a sequence of real numbers. Suppose that
	$$s_{k+1} \leq (1-a_{k}) s_{k} + a_{k} t_{k},\ \ \ \ \ \forall k\in \mathbb{N}$$
	If $\limsup t_{k_{n}} \leq 0$ for every subsequence $\{s_{k_{n}}\}$ of $\{s_k\}$ satisfying $\liminf (s_{k_{n}+1}-s_{k_{n}})\geq 0$, then $\lim s_{k} = 0$.
\end{lemma}

\begin{theorem}\label{strong convergence halpern}
Suppose that $f$ satisfies $P1, P3$, $P4^* $, $P4^{\bullet}$ and $S(f,K)\neq\varnothing$. If $f(\cdot,y)$ is $\bigtriangleup$-upper semicontinuous for all $y\in K$, then $\{x_k\}$ converges strongly to ${\rm Proj}_{S(f,K)}u$, where $\{x_k\}$ is a sequence generated by (\ref{bi-hc}).
\end{theorem}

\begin{proof}
Since $ S(f,K)$ is closed and convex, therefore we assume that  $ x^*={\rm Proj}_{S(f,K)}u$. Note that $f(y_k,x^*)+\lambda_{k-1}\langle \overrightarrow{x_{k-1}y_k},\overrightarrow{y_kx^*}\rangle \geq 0$. In turn, since by $P4^* $  every element of $S(f,K)$ solves $CFP(f,K)$, we have $f(y_k,x^*)\leq 0$, thus
$\langle \overrightarrow{x_{k-1}y_k},\overrightarrow{y_kx^*}\rangle \geq 0$, i.e. we have \begin{equation}d^2(x^*,x_{k-1})-d^2(x^*,y_k)-d^2(x_{k-1},y_k)\geq 0.\label{eq4}\end{equation}
Hence $d(x^*,y_k)\leq d(x^*,x_{k-1})$.
On the other hand by (\ref{bi-hc}), we have:\\
$d(x^*,x_k)\leq \alpha_k d(x^*,u)+(1-\alpha_k)d(x^*,y_k)
\leq \alpha_k d(x^*,u)+(1-\alpha_k)d(x^*,x_{k-1})\leq max \{d(x^*,u),d(x^*,x_{k-1})\}
 \leq \cdots \leq max \{d(x^*,u),d(x^*,x_0)\}$.\\
Therefore $\{x_k\}$ is bounded. Since $d(x^*,y_k)\leq d(x^*,x_{k-1})$ for all $k \in \mathbb{N}$, $\{y_k\}$ is bounded.
Now, by (\ref{bi-hc}), we have:\\
$d^2(x_{k+1},x^*)\leq (1-\alpha_{k+1})d^2(y_{k+1},x^*)+\alpha_{k+1}d^2(u,x^*)-\alpha_{k+1}(1-\alpha_{k+1})d^2(u,y_{k+1}).$
Since by (\ref{eq4}), $d^2(x^*,y_{k+1})\leq d^2(x^*,x_k)$, therefore
\begin{equation}
d^2(x_{k+1},x^*)\leq (1-\alpha_{k+1})d^2(x_k,x^*)+\alpha_{k+1}d^2(u,x^*)-\alpha_{k+1}(1-\alpha_{k+1})d^2(u,y_{k+1}).
\label{eq5}\end{equation}
In the sequel, we show $d(x_{k+1},x^*)\rightarrow 0$. By Lemma \ref{hss1}, it suffices to show that $\limsup (d^2(u,x^*)-(1-\alpha_{k_{n}+1})d^2(u,y_{k_{n}+1}))\leq 0$  for every subsequence $\{d^2(x_{k_{n}}, x^*)\}$ of $\{d^2(x_k,x^*)\}$ satisfying  $ \liminf (d^2(x_{k_{n}+1}, x^*) - d^2(x_{k_{n}}, x^*)) \geq 0$.\\
	For this, suppose that $\{d^2(x_{k_{n}}, x^*)\}$ is a subsequence of $\{d^2(x_k,x^*)\}$ such that $\liminf (d^2(x_{k_{n}+1}, x^*) - d^2(x_{k_{n}}, x^*)) \geq 0$. Then\\
$0\leq \liminf(d^2(x^*,x_{{k_n}+1})-d^2(x^*,x_{k_n}))\leq \liminf(\alpha_{{k_n}+1}d^2(x^*,u)+(1-\alpha_{{k_n}+1})d^2(x^*,y_{{k_n}+1}) \\ -d^2(x^*,x_{k_n}))
=\liminf(\alpha_{{k_n}+1}(d^2(x^*,u)-d^2(x^*,y_{{k_n}+1}))+d^2(x^*,y_{{k_n}+1}) -d^2(x^*,x_{k_n}))\\
\leq\limsup\alpha_{{k_n}+1}(d^2(x^*,u)-d^2(x^*,y_{{k_n}+1}))+\liminf (d^2(x^*,y_{{k_n}+1}) -d^2(x^*,x_{k_n}))\\
= \liminf (d^2(x^*,y_{{k_n}+1}) -d^2(x^*,x_{k_n}))\leq\limsup (d^2(x^*,y_{{k_n}+1}) -d^2(x^*,x_{k_n}))\leq0$.\\
Therefore, we conclude that $\lim (d^2(x^*,y_{{k_n}+1})-d^2(x^*,x_{k_n}))=0$, hence by (\ref{eq4}), we get $\lim d^2(x_{k_n},y_{{k_n}+1})=0$.\\
On the other hand, there are a subsequence $\{y_{{k_{n_i}}+1}\}$ of $\{y_{{k_n}+1}\}$ and $p\in K$ such that
$y_{{k_{n_i}}+1}\overset{\triangle}{\longrightarrow}p$  and
$$\limsup (d^2(u,x^*)-(1-\alpha_{{k_n}+1})d^2(u,y_{{k_n}+1}))=\lim (d^2(u,x^*)-(1-\alpha_{{k_{n_i}}+1})d^2(u,y_{{k_{n_i}}+1}))$$
Since $y_{{k_{n_i}}+1}\overset{\triangle}{\longrightarrow} p$. Hence, we have:\\
$0\leq\limsup (f(y_{{k_{n_i}}+1},y)+\lambda_{k_{n_i}}\langle \overrightarrow{x_{k_{n_i}}y_{{k_{n_i}}+1}},\overrightarrow{y_{{k_{n_i}}+1}y}\rangle )\\
\leq\limsup ( f(y_{{k_{n_i}}+1},y)+\lambda_{k_{n_i}}d(x_{k_{n_i}},y_{{k_{n_i}}+1}) d(y_{{k_{n_i}}+1},y))\leq f(p,y)$,\\
for all $y\in K$. Therefore $p\in S(f,K)$.
Now, since $ x^*={\rm Proj}_{S(f,K)}u$, hence  $$\limsup (d^2(u,x^*)-(1-\alpha_{{k_n}+1})d^2(u,y_{{k_n}+1}))\leq d^2(u,x^*)-d^2(u,p)\leq0.$$
Therefore Lemma \ref{hss1} shows that
\begin{equation}d(x_{n+1},x^*)\rightarrow 0.\label{rks}\end{equation}
Also, \eqref{eq4} implies that $d(y_{n+1}, x^*) \rightarrow 0$, which is the desired result.

\end{proof}

\section{Applications to Fixed Point Theory and Convex Minimization}

In this short section we present two examples of equilibrium problems in Hadamard spaces.

1. Let $X$ be a Hadamard space. $T:X\rightarrow X$ is called pseudo-contraction iff $$\langle \overrightarrow{TxTy},\overrightarrow{xy}\rangle\leq d^2(x,y),\ \ \forall x,y\in X.$$
It is easy to check that if $T$ is pseudo-contraction, then $f(x,y)=\langle\overrightarrow{Txx},\overrightarrow{xy}\rangle$ is a monotone bifunction. If $T$ is nonexpansive, which is a stronger condition, then $J^f_{\lambda}=J^T_{\lambda}$, where $J^T_{\lambda}$ is the resolvent of $T$ (see \cite{bac2, bac-rei, kr2}). Now the results of this paper is applicable to find and approximate an equilibrium point of $f$, which is a solution of variational inequality for $T$.

2. To solve the constraint minimization problem
$$\min_{x\in K} \varphi(x),$$ which the constraint set $K$ is a convex and closed subset of a Hadamard space $X$ and $\varphi:X\rightarrow]-\infty,+\infty]$ is a convex, proper and lower semicontinuous function, we can consider the monotone bifunction $f(x,y)=\varphi(y)-\varphi(x)$ on $K\times K$. It is easy to see that $J_{\lambda}^f=J_{\lambda}^{\varphi}$, where $J_{\lambda}^{\varphi}$ is the resolvent of $\varphi$. (see \cite{bac1, bac2}). Then the methods discussed in Sections 4-6 are applicable to approximate an equilibrium point of $f$, which is a minimum point of $\varphi$ on $K$. In fact Theorem \ref{strong convergence halpern} in case $f(x,y)=\varphi(y)-\varphi(x)$ extends Theorem 3.2 of \cite{ch}, which the sequence $\lambda_n$ is constant.

There are several examples of convex functions and minimization in Hadamard spaces. Some examples are energy functional on a Hadamard space and computation of median and means for a finite family of points in a Hadamard space. For more examples and explanations, the interest reader can consult \cite{bac1, bac2}.

%---------------------------------------------------------------------------------------%

%---------------------------------------------------------------------------------------%

\end{document}